\documentclass[11pt]{amsart}
\usepackage{latexsym, amsmath, amssymb, amsthm, mathrsfs,wasysym,mathtools}
\usepackage[alwaysadjust]{paralist}
\usepackage{caption}
\usepackage{subcaption}
\usepackage[T1]{fontenc}
\usepackage{lmodern}
\usepackage[backref,colorlinks=true,linkcolor=blue,urlcolor=blue, citecolor=blue]%
{hyperref}
\usepackage[shortalphabetic]{amsrefs}
\usepackage[all]{xy}
\usepackage[T1]{fontenc}
\usepackage{mathptmx}
\usepackage{extarrows}
\usepackage{microtype}
\usepackage{framed}
\usepackage{tikz}
\usepackage{tikz-cd}
\usepackage{graphicx}
\usepackage[alwaysadjust]{paralist}
\makeatletter
\providecommand\@enum@widestlabel{7}
\makeatother

\usepackage[centering, includeheadfoot, hmargin=.75in, vmargin=.75in,
headheight=30.4pt]{geometry}
\newtheorem{lemma}{Lemma}[section]
\newtheorem{theorem}[lemma]{Theorem}
\newtheorem{corollary}[lemma]{Corollary}
\newtheorem{proposition}[lemma]{Proposition}

\theoremstyle{definition}
\newtheorem{definition}[lemma]{Definition}
\newtheorem{remark}[lemma]{Remark}
\newtheorem{example}[lemma]{Example}
\newtheorem{question}[lemma]{Question}

\renewcommand{\theequation}%
{\arabic{section}.\arabic{lemma}.\arabic{equation}}

\newcommand{\sI}{\ensuremath{\kern -1pt \mathscr{I}\kern -2pt}} 
\newcommand{\sJ}{\ensuremath{\kern -2pt \mathscr{J}\kern -2pt}}

\newcommand{\bb}{\ensuremath{\mathfrak{b}}}

\renewcommand{\geq}{\geqslant}
\renewcommand{\leq}{\leqslant}

\DeclareMathOperator{\mult}{mult}
\DeclareMathOperator{\Nef}{Nef}

\DeclareMathOperator{\Sym}{Sym}

\newcommand{\ch}{\textup{ch}}
\newcommand{\lc}{\textup{lc}}

\input xy
\xyoption{all}

\usepackage{comment}
\usepackage{framed}
\usepackage{color}
\definecolor{shadecolor}{gray}{0.875}

\let\cal\mathcal

\let\bb\mathbb

\title{Positivity and base loci of vector bundles revisited}
\author{Mihai Fulger}
\address{Department of Mathematics, University of Connecticut, Storrs, CT 06269-1009
}
\address{Institute of Mathematics of the Romanian Academy, P.O. Box 1-764, RO-014700, Bucharest, Romania}
\email{mihai.fulger@uconn.edu}

\author{Nabanita Ray}

\address{Chennai Mathematical Institute, H1 SIPCOT IT Park, Siruseri, Kelambakkam 603103, 
	India}

\email{nabanitar@cmi.ac.in/nabanitaray2910@gmail.com}
\date{}

\begin{document}
	
	\maketitle
	
	\begin{abstract}
		We give equivalent descriptions for the augmented and diminished base loci of vector bundles in characteristic zero. We show that these base loci behave well under pullback, tensor product, and direct sum. Pathological behavior is observed on some nonsplit exact sequences. 
	\end{abstract}
	
	\section{Introduction}
	
	It is classical to study the geometry of a projective variety $X$ by using positivity properties of line bundles on it. The standard generalization to higher rank bundles $E$ is to impose the analogous positivity property to the Serre line bundle $\mathcal O_{\bb P(E)}(1)$ on the projective bundle of 1-dimensional quotients of $E$. This approach has been studied intensely for ampleness or nefness. Its generalizations to bigness and pseudo-effectivity are called \emph{L-bigness} and \emph{L-pseudo-effective} respectively.
	However there are stronger generalizations due to Viehweg that we define below.
	
	Fix any ample divisor polarization $A$ of $X$. For a vector bundle $E$ on $X$ we may define the \emph{augmented base locus} $\bb B_+(E)$ by
	\[X\setminus\bb B_+(E)=\bigl\{x\in X\ \mid\ \exists\  m,q\geq 1,\ \Sym^m(\Sym^qE(-A))\text{ is generated by global sections at } x\bigr\}.\]
	Then say that $E$ is \emph{Viehweg-big} (or \emph{V-big}) for short if $\bb B_+(E)\neq X$. In rank 1 this agrees with the usual bigness notion for line bundles. For ranks $r>1$ however it is a more restrictive notion than {L-bigness}. For instance a direct sum of line bundles on a curve is L-big precisely when one of the summands is ample, but V-big precisely when all the summands are ample.
	We may also define the \emph{diminished base locus} $\bb B_-(E)$ by 
	\[X\setminus\bb B_-(E)=\bigl\{x\in X\ \mid\ \forall\ q\geq 1\ \exists\ m\geq 1,\ \Sym^m(\Sym^qE(A))\text{ is generated by global sections at }x\bigr\}.\] 
	Say that $E$ is \emph{V-pseudo-effective} (\emph{V-psef} for short) if $\bb B_-(E)\neq X$. Again for line bundles this agrees with the usual pseudo-effectivity notion. However for ranks $r>1$ V-pseudo-effecitivity is usually stronger than L-pseudo-effectivity.
	
	These base loci were also introduced and studied in \cite{BKKMSU}. They prove that $\bb B_{\pm}(E)=\pi(\bb B_{\pm}(\cal O_{\bb P(E)}(1)))$, where $\pi:\bb P(E)\to X$ is the bundle projection. Our work complements theirs. 
	In our first result we offer equivalent descriptions for these notions of base locus.
	
	\begin{theorem}[\ref{lem:baselociviatensorpowers},\ref{remark:basesimplesym}]\label{thm:interpretations}
		Let $X$ be a projective variety over an algebraically closed field of characteristic zero.
		Fix $x\in X$ and $A$ an ample Cartier divisor on $X$. The following are equivalent:
		\begin{enumerate}
			\item There exist $q,m\geq 1$ such that $\bigotimes^{mq}E(-mA)$ is generated by global sections at $x$.
			\item $x\not\in\bb B_+(E)$.
			\item There exist $q,m\geq 1$ such that $\Sym^{mq}E(-mA)$ is generated by global sections at $x$.
		\end{enumerate}
	\end{theorem}
	
	\noindent Similar equivalent interpretations exist for $\bb B_-(E)$. The implications $(1)\to(2)\to (3)$ are straightforward without restriction on the characteristic. The idea of $(3)\to (1)$ is that in characteristic zero tensor powers are direct sums of Schur functors each of which is a direct summand of a tensor product of symmetric powers. This idea was also employed in \cite{FM2} to prove that $(1)$ and $(2)$ are equivalent, and earlier by \cite{Jager} to prove that if $E$ is an ample vector bundle, then $\bigotimes^mE$ is eventually globally generated.
	The arguments for $\bb B_-(E)$ are new.
	
	As a consequence we see in Corollary \ref{cor:homogeneousB} that $\bb B_{\pm}$ base loci are homogeneous for symmetric (or tensor) powers. In particular $E$ is V-big if and only if $\Sym^mE$ is V-big for some (or any) $m\geq 2$. In Example \ref{ex:Lcounter} we see that the converse fails for L-bigness.
	\smallskip
	
	We also study the behavior of base loci under pullback of vector bundles by surjective morphisms:
	
	\begin{theorem}[\ref{prop:B+pull},\ref{prop:restrictedbasepullback} ]Let $f:X\to Y$ be a surjective morphism of normal projective varieties over an algebraically closed field of characteristic zero and let $E$ be a vector bundle on $Y$. Then
		\begin{enumerate}
			\item We have $\bb B_+(f^*E)=f^{-1}\bb B_+(E)\bigcup\text{NF}(f)$, where $\text{NF}(f)$ is the union of the curves $C\subset X$ that are contracted by $f$.
			\item If $Y$ is smooth, or if $f$ is equidimensional, then $\bb B_-(f^*E)=f^{-1}\bb B_-(E)$.
		\end{enumerate}
	\end{theorem}
	
	\noindent In their generality, these results appear new even for line bundles. The focus in the literature has been on finite or birational morphisms. For instance \cite{DL21} prove the finite case of $(1)$ and $(2)$ for line bundles (or $\bb Q$-divisors) when $X$ and $Y$ are smooth using complex geometric methods, while \cite{BBP} prove with algebraic methods the birational case of $(1)$ for line bundles. Our methods are also algebraic. Furthermore we also consider $\bb R$-twisted bundles.
	\smallskip
	
	Finally we observe a pathology of augmented and diminished base loci. A vector bundle $E$ is ample (resp.~ nef) if and only if $\bb B_+(E)=\emptyset$ (resp.~$\bb B_-(E)=\emptyset$). This motivates calling $\bb B_+(E)$ (resp.~$\bb B_-(E)$) the \emph{nonample} (resp.~\emph{nonnef}) locus of $E$. In \cite{FM1} it is proved that an extension of ample (resp.~nef) coherent sheaves is again ample (resp.~nef). 
	In Lemma \ref{lem:directsums} we show that augmented and diminished base loci are ``additive'' for direct sums.
	These suggest that positivity of vector bundles does not worsen in extensions. Unfortunately this is not the case. 
	
	We construct Examples \ref{ex:B-huh} and \ref{ex:B+huh}
	of short exact sequences $0\to E'\to E\to E''\to 0$ on surfaces where $\bb B_{\pm}(E)\not\subset\bb B_{\pm}(E')\bigcup\bb B_{\pm}(E'')$. We have not found any example where $E'$ and $E''$ are V-big (resp.~V-psef), but $E$ is not. On the other hand, a bundle that contains an L-big bundle as a subsheaf is necessarily L-big (cf.~Proposition \ref{prop:subLbig}).
	\smallskip
	
	The last section is more informal. We explain how we can define cones of V-big and V-psef vector bundles inside the real numerical ring $N^*(X)$ of a complex projective manifold. 
	
	
	
	
	

	\subsection*{Acknowledgments} 
	We thank C.~Jager, S.~Kov\' acs, E.~Mistretta, T.~Murayama for useful conversations. The first named author was partially supported by the Simons Foundation Collaboration Grant 579353.

	\section{Augmented and diminished base loci}
	
	\subsection{Preliminaries}
	Let $X$ be a projective variety over an algebraically closed field of characteristic zero and let $E$ be a locally free sheaf (or
	vector bundle) over $X$. For a point $x \in X$, $E_x = E \otimes_{\mathcal{O}_X}\mathcal{O}_{X,x}$ denotes the stalk of $E$ at the point $x$ and
	$E(x) = E \otimes_{\mathcal{O}_X}\kappa(x)$ the fiber of $E$ at $x$, where $\kappa(x)$ is the residue field of $x$. Let $\pi:\bb P(E)\to X$ be the projective bundle of rank 1 quotients of the fibers $E(x)$. Let $\cal O_{\bb P(E)}(1)$ be the relative Serre bundle. We have $\pi_*\cal O_{\bb P(E)}(m)=S^mE$, where we write $S^mE=\Sym^mE$ for short. Denote $\xi\coloneqq c_1(\cal O_{\bb P(E)}(1))$. Also denote $T^nE\coloneqq\bigotimes^nE=E^{\otimes n}$.
	
	\begin{definition}[Base loci]
		The \bf base locus \rm of $E$ is
		$$\text{Bs}(E) =\{ x \in X \mid H^0 (X, E)\otimes\cal O_X\overset{ev}{\rightarrow} E(x) \text { is not surjective } \},$$
		the support of the cokernel $Q$ of the evaluation map $ev$. The \bf stable base locus \rm of $E$ is 
		$$\mathbb{B}(E)=\bigcap\limits_{m>0}\text{Bs}(S^mE).$$
		It is equal to $\text{Bs}(S^mE)$ for all sufficiently divisible positive $m$.
		Let $A$ be any ample line bundle on $X$. Then the \bf augmented base locus \rm of $E$ is
		$$\mathbb{B}_+(E)=\bigcap_{q\in\bb N}\bb B(S^qE\otimes A^{-1})=\bigcap\limits_{q\in \mathbb{N}}\bigcap\limits_{m\in \mathbb{N}}\text{Bs}(S^{m}S^{q}E\otimes A^{-m}).$$
		The \bf diminished base locus \rm of $E$ is
		$$ \mathbb{B}_-(E)=\bigcup\limits_{q\in\bb N}\bb B(S^qE\otimes A)=\bigcup\limits_{q\in \mathbb{N}}\bigcap\limits_{m\in \mathbb{N}}\text{Bs}(S^{m}S^qE\otimes A^m).$$
	\end{definition}
	
	\noindent The definitions are independent of the choice of the polarization $A$. Note that $\mathbb{B}_+(E)$ is always a closed subset of $X$ but already for line bundles it is not known whether $\mathbb{B}_-(E)$ is closed. \cite{Le} gives an example of an $\bb R$-divisor class whose diminished base locus is dense and not closed, Again, no such example of a Cartier $\bb Z$-divisor seems to be known.
	
	For line bundles $L$, the augmented and diminished base loci are homogeneous numerical invariants that are furthermore well defined on the real N\' eron--Severi space $N^1(X)$.
	
	\begin{definition}\label{VandLPositiveDefinition}
		Say that $E$ is 
		\begin{itemize}
			\item \bf V-big (or Viehweg-big) \rm if  $\mathbb{B}_+(E) \neq X$.
			\item \bf L-big \rm if $\mathcal{O}_{\mathbb{P}(E)}(1)$ is big on $\mathbb{P}(E)$, i.e., $\mathbb{B}_+(\mathcal{O}_{\mathbb{P}(E)}(1))\neq \mathbb{P}(E)$.
			\item \bf V-psef (or Viehweg-pseudo-effective) \rm if $\mathbb{B}_-(E) \neq X$.
			\item \bf L-psef \rm if $\mathcal{O}_{\mathbb{P}(E)}(1)$ is pseudo-effective on $\mathbb{P}(E)$, i.e., $\mathbb{B}_-(\mathcal{O}_{\mathbb{P}(E)}(1))\neq \mathbb{P}(E)$.
		\end{itemize}
		We use {\bf V-positive} as a fixed choice of either V-big or V-psef. Similarly we use {\bf L-positive}.
	\end{definition}
	
	\noindent Following \cite{BKKMSU}, since the diminished base locus is not known to be closed, we can also define {\bf weak positivity} by the requirement that the diminished base locus is not dense in $X$ (resp.~$\bb P(E)$). This positivity notion is weaker than bigness and possibly stronger than pseudo-effectivity. We will not use it in our work.

	\begin{remark}\label{remark:easybase}
		Let $E$ and $F$ be two vector bundles. Then  
		\begin{enumerate}[(1)]
			\item $\text{Bs}(E\oplus F)=\text{Bs}(E)\bigcup \text{Bs} (F)$;
			\item $\text{Bs}(E\otimes F)\subseteq\text{Bs}(E)\bigcup\text{Bs}(F)$;
			\item $\text{Bs}(E)\supseteq\text{Bs}(E^{\otimes m})\supseteq\text{Bs}(S^mE)$ for $m\geq 1$.
		\end{enumerate}
	\end{remark}

	\begin{lemma}\label{lem:quotients}
		Let $f:E\rightarrow F$ be a map between two vector bundles, surjective over an open subset $U\subseteq X$. Then
		\[\text{Bs}(F)\cap U\subseteq\text{Bs}(E)\cap U\quad\text{ and }\quad\bb B(F)\cap U\subseteq\bb B(E)\quad\text{ and }\quad\bb B_{\pm}(F)\cap U\subseteq\bb B_{\pm}(E)\cap U\] In particular, if $E$ is V-positive, then $F$ is also V-positive.
	\end{lemma}
		\begin{proof}The first inclusion is clear. Tensor products and symmetric powers preserve surjections. The other inclusions then follow from the first.
		\end{proof}

		\begin{example}
			If $X$ is a curve and $\bb B_-(E)\neq X$, then $E$ is nef, so $\bb B_-(E)=\emptyset$. Similarly, if $\bb B_+(E)\neq X$, then $E$ is ample and $\bb B_+(E)=\emptyset$. In particular on curves V-bigness is equivalent to ampleness and V-pseudo-effectivity is equivalent to nefness.
			(Indeed if some $x$ is not in $\bb B_-(E)$, then for all $q\geq 1$ there exists $m=m(E,q)$ such that $S^{mq}E\otimes A^m$ is globally generated at $x$. Since $X$ is a curve, it follows that $S^{mq}E\otimes A^m$ is nef, for instance by \cite{L2}*{Example 6.4.17}. By using the relative Veronese embedding $\bb P(E)\subset\bb P(S^{mq}E)$, we deduce that $mq\xi+m\pi^*A$ is nef for all $q\geq 1$. Since the nef cone of $\bb P(E)$ is closed, we obtain that $\xi$ is nef, thus so is $E$. See also \cite{L2}*{Example 6.2.13}.)
			\qed
		\end{example}
		
		\begin{remark}
			The situation for L-bigness is very different. A direct sum $E=\bigoplus_{i=1}^rB_i$ of line bundles on a curve is L-big if and only if one of the $B_i$ is ample, and it is V-big if and only if all the $B_i$ are ample.
		\end{remark}

		\subsection{Base loci via tensor powers}\label{sec:baselocitensorpowers}
		In general $S^{mn}E\not\simeq S^mS^nE\not\simeq S^nS^mE$. This makes it nontrivial to translate between positivity properties of $E$ and positivity properties of $S^nE$.
		This issue is also observed in \cite{MU}. To bypass it, we use tensor powers $T^nE=\bigotimes^nE=E^{\otimes n}$. These do satisfy $T^{mn}E\simeq T^nT^mE\simeq T^mT^nE$. The price to pay is that they are not related by one geometric object like $\bb P(E)$. 
		
		\begin{definition}\label{definition_of_bl_tensor_power}
			Put $\bb B^{\otimes}(E)\coloneqq\bigcap_{m\geq 1}\text{Bs}(T^mE)$. For $A$ ample put 
			\[\bb B_{+,A}^{\otimes}\coloneqq\bigcap_{q\geq 1}\bigcap_{m\geq 1}\text{Bs}(T^{mq}E\otimes A^{-m})=\bigcap_{q\geq 1}\bb B^{\otimes}(T^qE\otimes A^{-1})\]
			and 
			\[\bb B^{\otimes}_{-,A}(E)\coloneqq\bigcup_{q\geq 1}\bigcap_{m\geq 1}\text{Bs}(T^{mq}E\otimes A^m)=\bigcup_{q\geq 1}\bb B^{\otimes}(T^qE\otimes A).\]
		\end{definition}
		
		\noindent For sufficiently divisible $m$ we have that $\bb B^{\otimes}(E)=\text{Bs}(T^mE)$ by noetherianity. In particular $\bb B^{\otimes}(E)=\bb B^{\otimes}(T^nE)$ for all $n\geq 1$.
		Due to the natural surjections $T^{mq}E\to S^mS^qE$, we have $\bb B(S^qE)\subseteq\bb B^{\otimes}(T^qE)=\bb B^{\otimes}(E)$.
		Furthermore $\bb B_{\pm}(E)\subseteq\bb B^{\otimes}_{\pm,A}(E)$ by Lemma \ref{lem:quotients}.
		We will prove that (at least in characteristic 0) the inclusion is an equality. First we check that the definition is independent of $A$.
		
		\begin{remark}\label{remark:independentofample}
			If $A$ and $B$ are ample divisors, then $\bb B_{\pm,A}^{\otimes}(E)=\bb B_{\pm,B}^{\otimes}(E)$.
			(Fix $p$ such that $pB-A$ is globally generated. For $q\geq 0$ fix $m=m(q)$ sufficiently divisible such that $\text{Bs}(T^{m\cdot(pq)}E\otimes A^m)=\bigcap_{m\geq 1}\text{Bs}(T^{m\cdot (pq)}E\otimes A^m)$ and $\text{Bs}(T^{(mp)q}E\otimes B^{mp})=\bigcap_{m\geq 1}\text{Bs}(T^{mq}E\otimes B^m)$. Then $\text{Bs}(T^{mpq}E\otimes B^{mp})\subseteq \text{Bs}(T^{m(pq)}E\otimes A^m)$ by Remark \ref{remark:easybase}.  
			Taking unions over all $q$, we obtain $\bb B^{\otimes}_{-,B}(E)\subseteq\bb B^{\otimes}_{-,A}(E)$. The reverse inclusion follows by swapping $A$ and $B$. The case of $\bb B_+$ is similar and also done in \cite{FM2}.)
		\end{remark}
		
		\begin{definition}
			Denote then the common value $\bb B_{\pm}^{\otimes}(E)\coloneqq\bb B^{\otimes}_{\pm,A}(E)$.
		\end{definition}

		\begin{lemma}\label{lem:baselociviatensorpowers}$\bb B_{\pm}(E)=\bb B_{\pm}^{\otimes}(E)$.
		\end{lemma}
		\begin{proof}
			\cite{FM2} proved the statement about augmented base loci. The argument for the diminished base loci is similar as we show here.
			Since $S^{m}S^qE$ is a quotient of $T^{mq}E$, we deduce that $\bb B_-(E)\subseteq \bb B^{\otimes}_-(E)$ by Remark \ref{remark:easybase}. Conversely, assume that $x\not\in \bb B_{-}(E)$. Then for every $q\geq 1$ there exists $m=m(q)\geq 1$ such that $S^{m}S^qE\otimes A^m$ is generated by global sections at $x$. In particular $S^{mq}E\otimes A^m$ is also generated by global sections at $x$. Using $\bb B_{-,A}^{\otimes}(E)=\bb B_{-,2A}^{\otimes}(E)$ from Remark \ref{remark:independentofample}, we want to prove that for all $q$ we can find $M=M(q)$ such that $T^{M\cdotp q}E\otimes A^{2M}$ is globally generated at $x$. 
			Since we restrict ourselves to characteristic zero, we have a direct sum decomposition \[T^{Mq}E=\bigoplus_{\lambda}\bb S_{\lambda}E\]
			where $\bb S_{\lambda}E$ denotes the Schur functor corresponding to the partition $\lambda=(\lambda_1\geq\lambda_2\geq\ldots)$ with $\lambda_1+\lambda_2+\ldots=Mq$. Put $r\coloneqq {\rm rk}\, E$. If $\lambda_{r+1}>0$, then $\bb S_{\lambda}E=0$. Thus we can restrict our attention to partitions $\lambda$ with at most $r$ nonzero parts $\lambda_i$.
			In characteristic zero we have that $\bb S_{\lambda}E$ is also a direct summand of $S^{\lambda_1}E\otimes\ldots\otimes S^{\lambda_r}E$ by Pieri's rule (\cite{FH}*{(6.8) p.~79)}).
			
			Write $\lambda_i=a_i(mq)+b_i$ with $0\leq b_i<mq$. Then $S^{\lambda_1}E\otimes\ldots\otimes S^{\lambda_r}E$ is a direct summand of $S^{a_1mq}E\otimes\ldots\otimes S^{a_rmq}E\otimes(S^{b_1}E\otimes\ldots\otimes S^{b_r}E)$ by Pieri's rule.
			Let $N=N(q)$ be sufficiently large so that $S^{b_1}E\otimes\ldots\otimes S^{b_r}E\otimes A^n$ is globally generated for all possible $(b_1,\ldots,b_r)\in\{0,\ldots,mq-1\}^r$ and all $n\geq N$. We have
			\[S^{a_1mq}E\otimes\ldots\otimes S^{a_rmq}E\otimes(S^{b_1}E\otimes\ldots\otimes S^{b_r}E)\otimes A^{2M}=\bigl(\bigotimes_{i=1}^r(S^{a_imq}E\otimes A^{a_im})\bigr)\otimes(S^{b_1}E\otimes\ldots\otimes S^{b_r}E\otimes A^{2M-m\sum_ia_i}).\]
			Since $2M-m\sum_ia_i=M+\frac 1q\sum_ib_i\geq M$, any value $M\geq N$ will do.
		\end{proof}

		\begin{remark}\label{remark3.8}\label{remark:basesimplesym}Similarly
			\[\bb B_{-}(E)=\bigcup_{q\geq 1}\bigcap_{m\geq 1}\text{Bs}(S^{mq}E\otimes A^m)\text{ and }\bb B_+(E)=\bigcap_{q\geq 1}\bigcap_{m\geq 1}\text{Bs}(S^{mq}E\otimes A^{-m}).\] 
			These are closer in spirit to \cite{MU}*{Definition 2.4}.
			(Since $S^{mq}E$ is a quotient of $S^mS^qE$, the RHS of both claimed equalities is clearly contained in the corresponding LHS. For fixed $q$, the intersection $\bigcap_{m\geq 1}\text{Bs}(S^{mq}E\otimes A^{\pm m})$ equals $\text{Bs}(S^{mq}E\otimes A^{\pm m})$ for sufficiently divisible $m$.
			If $x$ is not in the RHS, then for all $q\geq 1$ there exists $m$ such that $S^{mq}E\otimes A^{\pm m}$ is globally generated at $x$. As in the proof of Lemma \ref{lem:baselociviatensorpowers} (respectively using \cite{FM2} for the augmented base locus) we find $M$ such that $T^{Mq}E\otimes A^{\pm 2M}$ is globally generated at $x$, hence so is its quotient $S^MS^qE\otimes A^{\pm 2M}$.)
		\end{remark}
		
		\begin{corollary}[Homogeneity]\label{cor:homogeneousB}
			We have $\bb B^{\otimes}(E)=\bb B^{\otimes}(T^cE)$, and 
			\[\bb B_{\pm}(E)=\bb B_{\pm}(S^cE)=\bb B_{\pm}(T^cE)\] for all $c\geq 1$.
			In particular $T^cE\text{ is V-positive}\Leftrightarrow E\text{ is V-positive}\Leftrightarrow S^cE\text{ is V-positive}$ for any $c\geq 1$. 
		\end{corollary}
		\begin{proof}Both $\bb B^{\otimes}(E)$ and $\bb B^{\otimes}(T^cE)$ equal $\text{Bs}(T^mE)$ for $m$ sufficiently divisible, in particular divisible by $c$. All homogeneity properties for $T^cE$ follow.
			From the surjections $T^cE\to S^cE$ we deduce by Lemma \ref{lem:quotients} that $\bb B_{\pm}(S^cE)\subseteq\bb B_{\pm}(T^cE)=\bb B_{\pm}(E)$. The reverse inclusion is a consequence of the surjections $T^{mq}S^cE\to S^{(mc)q}E$ and Remark \ref{remark:basesimplesym}.
		\end{proof}
		
		\begin{corollary}\label{cor:inclusionsvariousloci}
			We have inclusions
			\[\bb B_-(E)\subseteq\bb B(E)\subseteq\bb B_+(E).\]
			In particular, if $E$ is V-big, then it is also V-psef.
		\end{corollary}
		\begin{proof}
			For every $q\geq 1$ and very ample $A$ we can find $n$ sufficiently divisible such that $\bigcap_{m\geq 1}{\rm Bs}(S^{mq}E\otimes A^{\pm m})={\rm Bs}(S^{nq}E\otimes A^{\pm n})$ and such that $\bb B(E)={\rm Bs}(S^{nq}E)$. Clearly ${\rm Bs}(S^{nq}E\otimes A^n)\subseteq{\rm Bs}(S^{nq}E)\subseteq{\rm Bs}(S^{nq}E\otimes A^{-n})$. Conclude by Remark \ref{remark:basesimplesym}.
		\end{proof}

		\begin{lemma}
			Let $\lambda$ be any partition with at most $r={\rm rk}\, E$ parts. If $E$ is a V-positive vector bundle, then so are all Schur functors $\bb S_{\lambda}E$, in particular so is $\det E$.
			\begin{proof}
				We have surjective maps $T^{r}E\rightarrow \bb S_{\lambda}E$. 
				The result follows from Lemma \ref{lem:quotients}. Note that $\det E=S_{(1^r)}E$.
			\end{proof}
		\end{lemma}

				\subsection{Positvity via the Serre line bundle} 
				
				Classically, positivity properties like ampleness or nefness for $E$ can be defined by the corresponding property for $\xi$. We have seen that the situation is not the same for (V-)bigness. Here we investigate the relation between base loci of $E$ and of $\mathcal O_{\bb P(E)}(1)$.
				
				\begin{lemma}
					Let $x\in X$. Then $\mathcal O_{\bb P(E)}(1)$ is generated by global sections at all points of $\pi^{-1}\{x\}=\bb P(E(x))\subset\bb P(E)$ if and only if $E$ is globally generated at $x$. In particular $\text{Bs}(E)=\pi(\text{Bs}(\mathcal O_{\bb P(E)}(1)))$. 
				\end{lemma}
				\begin{proof}
					If $E$ is globally generated at $x$, then $\pi^*E$ is globally generated along $\bb P(E(x))$, hence so is its quotient $\mathcal O_{\bb P(E)}(1)$. 
					Conversely, if $\mathcal O_{\bb P(E)}(1)$ is globally generated along $\bb P(E(x))$,
					consider the evaluation sequence $0\to K\to H^0(X,E)\otimes\cal O_{\bb P(E)}\to \cal O_{\bb P(E)}(1)$ on $\bb P(E)$. By assumption there exists $x\in W\subset X$ an open neighborhood such that the evaluation sequence is also exact on the right over $\pi^{-1}W$. For $z\in W$ we have that the sequence restricts as the exact sequence $0\to K|_{\bb P(E(z))}\to H^0(X,E)\otimes\cal O_{\bb P(E(z))}\to \cal O_{\bb P(E(z))}(1)\to 0$. The resulting map $H^0(X,E)\to H^0(\bb P(E(z)),\cal O_{\bb P(E(z))}(1))$ must be surjective since no proper subspace of $H^0(\bb P^n,\mathcal O_{\bb P^n}(1))$ generates $\cal O_{\bb P^n}(1)$. Note that this fails for $\cal O_{\bb P^n}(m)$ with $m>1$\footnote{This fact was missed in the proof of \cite{FM1}*{Proposition 6.4}, but the proof there can be corrected with the arguments of \cite{BKKMSU}*{Proposition 3.1 and 3.2}}. We deduce that $H^1(\bb P(E(z)),K|_{\bb P(E(z))})=0$ and $R^1\pi_*K$ is supported away from $W$. By pushing forward the evaluation sequence we deduce that $E$ is globally generated at $x$. 
				\end{proof}
				
				\begin{remark}
					If $L$ is a $\pi$-ample line bundle on $\bb P(E)$, then we are usually only guaranteed $\pi(\text{Bs}(L))\subset\text{Bs}(\pi_*L)$.
					\cite{MU}*{Example 3.2} provide an example where $\cal O_{\bb P(E)}(m)$ is globally generated for some $m\geq 1$, but $S^mE$ is never globally generated. In particular the inclusions $\pi(\text{Bs}(\cal O_{\bb P(E)}(m)))\subset\text{Bs}(S^mE)$ and $\pi(\bb B(\cal O_{\bb P(E)}(1)))\subset\bb B(E)$ can be strict. 
				\end{remark}
				
				This issue goes away when we allow perturbations:
				
				\begin{proposition}[\cite{BKKMSU}*{Proposition 3.1 and Proposition 3.2}]\label{prop:imageBpm} With assumptions as above,
					\[\pi(\bb B_{\pm}(\xi))=\bb B_{\pm}(E)\]
					In particular, if $E$ is V-positive, then it is also L-positive.
				\end{proposition}
				
				\noindent \cite{BKKMSU} assume smoothness for $X$, but this is not used in the proof. Note that they implicitly use the interpretations in Remark \ref{remark:basesimplesym} for $\bb B_{\pm}(E)$.

				\subsection{Positivity of tensor products and of direct sums}\label{sec:tensoranddirectsum}
				
				\begin{lemma}\label{lem:tensorbasepm}If $E$ and $F$ are vector bundles on $X$, then 
					\[\bb B_{\pm}(E\otimes F)\subseteq\bb B_{\pm}(E)\bigcup\bb B_{-}(F).\]
					It follows that the V-positivity notions are preserved by tensor product. Furthermore, if $E$ is V-big and $F$ is V-psef, then $E\otimes F$ is a V-big vector bundle. 
				\end{lemma}
				\begin{proof}
					We have $\bb B_-(E)=\bb B_-^{\otimes}(E)=\bigcup_{q\geq 1}\bb B^{\otimes}(T^qE\otimes A)$ by Lemma \ref{lem:baselociviatensorpowers}. For fixed $q$, choose $m=m(q)$ sufficiently divisble so that $\text{Bs}(T^{mq}E\otimes A^m)=\bb B^{\otimes}(T^qE\otimes A)$, and
					$\text{Bs}(T^{mq}F\otimes A^m)=\bb B^{\otimes}(T^qF\otimes A)$, and $\text{Bs}(T^{mq}(E\otimes F)\otimes A^{2m})=\bb B^{\otimes}(T^q(E\otimes F)\otimes A^2)$.
					Conclude the $\bb B_{-}$ case from the inclusion $\text{Bs}(T^{mq}(E\otimes F)\otimes A^{2m})\subseteq \text{Bs}(T^{mq}E\otimes A^m)\bigcup \text{Bs}(T^{mq}F\otimes A^m)$ given by Remark \ref{remark:easybase}. The statement about augmented base loci is analogous from the inclusion $\text{Bs}(T^{mq}(E\otimes F)\otimes A^{-m})\subseteq\text{Bs}(T^{mq}E\otimes A^{-2m})\bigcup\text{Bs}(T^{mq}F\otimes A^m)$.
				\end{proof}
				
				\begin{proposition}\label{prop:goodforcones}
					Let $E$ be a vector bundle on a smooth projective variety $X$ polarized by some ample divisor $A$. Then the
					following properties hold:
					\begin{enumerate}[(i)]
						\item $E$ is V-big if and only if there exists $c > 0$ such that $S^c E(-A)$ (or $T^c E(-A)$) is V-psef. 
						\item $E$ is V-psef if and only if $S^c E(A)$ (or $T^cE(A)$) is V-big for all $c>0$.
					\end{enumerate}
					\begin{proof}
						(i) If $S^cE\otimes A^{-1}$ is V-psef, then $S^cE$ is V-big and we deduce that $E$ is also V-big by Corollary \ref{cor:homogeneousB}.
						Conversely, if $E$ is V-big, then for some large $m$ we have that $\bb B(S^mE\otimes A^{-1})\neq X$. In particular $\bb B_-(S^mE\otimes A^{-1})\neq X$ by Corollary \ref{cor:inclusionsvariousloci}.
						
						(ii) If $E$ is V-psef, then so is $S^cE$ and then $S^cE\otimes A$ is V-big.
						Conversely, if $E$ is not V-psef, then for some $q\geq 1$ we have that $\bb B(S^qE\otimes A)=X$, hence by Corollary \ref{cor:inclusionsvariousloci} $\bb B_+(S^qE\otimes A)=X$ and $S^qE\otimes A$ is not V-big.
					\end{proof}
				\end{proposition}
				
				For direct sums we prove the following:
				
				\begin{lemma}\label{lem:directsums}
					Let $E$ and $F$ be two vector bundles on $X$. Then
					\[\mathbb{B}_{\pm}(E\oplus F)=\mathbb{B}_{\pm}(E)\bigcup\mathbb{B}_{\pm}(F)\] 
					In particular $V$-positivity is preserved by direct sums.
					\begin{proof}
						Note that $T^{ mq}(E\oplus F)=\bigoplus_{i+j=mq} T^iE\otimes T^jF$.
						Then $\text{Bs}(T^{ mq}(E\oplus F)\otimes A^{\pm m})\supseteq\text{Bs}(T^{ mq}E\otimes A^{\pm m})\bigcup \text{Bs}(T^{ mq}F\otimes A^{\pm m})$ by Remark \ref{remark:easybase}, thus $\mathbb{B}_{\pm}(E)\bigcup\mathbb{B}_{\pm}(F)\subseteq \mathbb{B}_{\pm}(E\oplus F)$. Conversely we first treat $\bb B_-$. Assume $x\not\in\bb B_-(E) \bigcup \bb B_-(F)$. Then for all $q$ we can find $m=m(q)$ such that $T^{mq}E\otimes A^m$ and $T^{mq}F\otimes A^m$ are globally generated at $x$.
						We want $M=M(q)$ such that $T^{M\cdotp q}(E\oplus F)\otimes A^{2M}$ is also globally generated at $x$, equivalently $T^iE\otimes T^jF\otimes A^{2M}$ is globally generated at $x$ for all $i+j=M\cdotp q$.
						
						Let $N=N(q)$ be sufficiently large so that $T^iE\otimes T^jF\otimes A^n$ is globally generated for all $(i,j)\in\{0,\ldots,mq-1\}^2$ and all $n\geq N$. For any $i\geq 0$ write $i=a_i(mq)+b_i$ with $0\leq b_i<mq$. Then $T^iE\otimes T^jF\otimes A^{2M}=T^{a_i}(T^{mq}E\otimes A^m)\otimes T^{a_j}(T^{mq}F\otimes A^m)\otimes(T^{b_i}E\otimes T^{b_j}F\otimes A^{2M-m(a_i+a_j)})$ is globally generated at $x$ for any choice $M\geq N$ since $m(a_i+a_j)=\frac 1q(i+j)-\frac 1q(b_i+b_j)\leq M$.
						
						If $x\not\in\bb B_+(E)\bigcup\bb B_+(F)$, then by choosing $A$ very ample we can find $q$ sufficiently divisible so that $T^qE\otimes A^{-2}$ and $T^qF\otimes A^{-2}$ are globally generated at $x$. As in the case of $\bb B_-$ we can find $M=M(q)$ such that $T^iE\otimes T^jF\otimes A^{-M}$ is globally generated at $x$ whenever $i+j=Mq$.  
					\end{proof}
				\end{lemma}
				
				An extension of globally generated vector bundles is not necessarily globally generated. 
				This is an illustration of the necessity of exactness on the right in the top row in the Snake Lemma.
				For instance if $E$ is an elliptic curve and $0\to\cal O_E\to \cal P\to\cal O_E\to 0$ is a nontrivial extension, then $H^0(E,\cal P)$ is 1-dimensional, hence $\cal P$ is not globally generated. It is however nef. More generally an extension of nef vector bundles is nef by \cite{FM1}*{Remark 3.4 and Lemma 3.31}.
				
				For line bundles, the diminished base locus is also known as the nonnef locus. This and the extension remarks above might suggest that if $0\to E'\to E\to E''\to 0$ is a short exact sequence of bundles, then $\bb B_-(E)\subseteq\bb B_-(E')\bigcup\bb B_-(E'')$, i.e., if $E'$ and $E''$ are ``nef at a point $x$'', then $E$ is also ``nef at $x$''. This is not the case:
				
				\begin{example}\label{ex:B-huh}
					Let $X$ be the surface obtained by blowing-up $\bb P^2$ in a point $p$ and then again in an infinitely near point $p'$. Denote by $L$ the pullback of a line, by $\overline F$ the strict transform of the first exceptional divisor, and by $F'$ the second exceptional divisor. We have an exact sequence
					\[0\to\cal O_X(L+\overline F)\to\cal O_X(L+\overline F+F')\to\cal O_{F'}(L+\overline F+F')\to 0\]
					The sums $L+(\overline F)$ and $L+(\overline F+F')$ are Zariski decompositions on $X$, and 
					$\cal O_{F'}(L+\overline F+F')$ is trivial (on $F'$). We have $\bb B_-(L+\overline F)=\overline F$ and $\bb B_-(L+\overline F+F')=\overline F+F'$ are the supports of the negative parts in the Zariski decompositions by \cite{ELMNP}*{Example 1.17}. This is not yet our desired example since the quotient is a torsion sheaf, not a vector bundle.
					
					The exact sequence above induces a commutative diagram with exact rows
					\[ \xymatrix{
						0\ar[r]& \cal O_X(L+\overline F)\ar[r]\ar@{=}[d]& E\ar[r]\ar@{->>}[d]&\cal O_X\ar[r]\ar@{->>}[d]&0\\
						0\ar[r]& \cal O_X(L+\overline F)\ar[r]& \cal O_X(L+\overline F+F')\ar[r]&\cal O_{F'}\ar[r]&0}
					\]
					where $E$ is a rank 2 vector bundle.
					We have $\bb B_-(L+\overline F)=\overline F$ and $\bb B_-(\cal O_X)=\emptyset$, however using Lemma \ref{lem:quotients} we obtain $\overline F+F'=\bb B_-(L+\overline F+F')\subseteq\bb B_-(E)$ properly contains their union.    \qed
				\end{example}
				
				Of course we also do not expect this inclusion to always hold. An easy example is given by the Euler sequence $0\to\cal O_{\bb P^1}(-1)\to\cal O_{\bb P^1}^2\to\cal O_{\bb P^1}(1)\to 0$.
				Similarly, the example below suggests that the augmented base locus of a vector bundle of rank bigger than 1 should not be called the nonample locus.

				\begin{example}\label{ex:B+huh}
					With $X$ the double blow-up of $\bb P^2$ in the example above, consider $C$ the strict transform of a conic in $\bb P^2$ passing through $p$ with tangent direction not corresponding to $p'$. Consider the exact sequence
					\[0\to\cal O_X(C+2F')\to\cal O_X(C+\overline F+2F')\to\cal O_{\overline F}(C+\overline F+2F')\to 0.\]
					We have $(C\cdot F')=0$, $({F'}^2)=-1$, and $(C^2)=3$, thus $C$ is the positive part of the Zariski decomposition of $C+2F'$. Furthermore $(C\cdot \overline F)=1$, hence $\overline F\not\subset\bb B_+(C+2F')$ by \cite{ELMNP}*{Example 1.11}.
					
					We compute $((C+\overline F+2F')\cdot \overline F)=1$, thus $\cal O_{\overline F}(C+\overline F+2F')\simeq\cal O_{\bb P^1}(1)$ is ample. 
					
					The divisor $P=C+\overline F+F'$ is nef because $C$ is nef, and $(P\cdot\overline F)=0$ and $(P\cdot F')=0$. In fact it is the pullback of the conic from $\bb P^2$. Since $({F'}^2)=-1$, these also imply that $P$ is the positive part of the Zariski decomposition of $C+\overline F+2F'$, and that $\overline F\subseteq\bb B_+(P+F')$ by \cite{ELMNP}*{Example 1.11}, since $(P\cdot \overline F)=0$.
					
					We look for an ample divisor $A$ on $X$ such that $(A\cdot\overline F)=1$. Assuming the existence of such a divisor, as above we construct a rank 2 vector bundle $E$ sitting in an extension
					\[0\to\cal O_X(C+2F')\to E\to\cal O_X(A)\to 0\]
					and that has $\cal O_X(C+\overline F+2F')$ as a quotient. In particular $\overline F\subset\bb B_+(C+\overline F+2F')\subseteq\bb B_+(E)$ by Lemma \ref{lem:quotients}. However $\overline F$ is not a subset of $\bb B_+(C+2F')$ or $\bb B_+(A)=\emptyset$. This would give our example. 
					
					To construct $A$, consider first the divisor $-2\overline F-3F'$. It is relatively ample over $\bb P^2$ and it has degree $1$ on $\overline F$. Then for all large enough $a>0$, with $L$ denoting the pullback of a general line in $\bb P^2$, the divisor $A=aL-2\overline F-3F'$ is ample and $(A\cdot\overline F)=1$.\qed
				\end{example}
				
				\begin{question} Are there extensions $0\to E'\to E\to E''\to 0$ of vector bundles such that $E'$ and $E''$ are V-big (resp.~V-psef), but $E$ is not V-big (resp.~V-psef)?
				\end{question}
				
				\subsection{Twisted bundles}
				Let $E$ be a vector bundle and let $T$ be an $\bb R$-Cartier $\bb R$-divisor on $X$. 
				Say $(E,T)$ and $(E',T')$ are equivalent if there exists a Cartier divisor $T''$ such that $E'=E(T'')$ and $T'=T-T''$.   
				The equivalence class of a pair $(E,T)$ is called the {\bf formal twist} of $E$ by $T$ and denoted $E\langle T\rangle$. We define $\bb P(E\langle T\rangle)$ to be $\bb P(E)$ polarized by $\cal O_{\bb P(E)}(1)\langle\pi^*T\rangle$. When $T$ is a $\bb Q$-Cartier $\bb Q$-divisor, we say that $E\langle T\rangle$ is $\bb Q$-{\bf twisted}. By abuse, on occasion we refer to a twisted bundle as $E$ instead of $E\langle T\rangle$. When $T=0$ we say that $E$ is an ({\bf untwisted}) bundle. 
				
				The theory of twisted bundles has natural dominant pullbacks, tensor products, and Schur functors, e.g., $\bb S_{\lambda}(E\langle T\rangle)=(\bb S_{\lambda}E)\langle mT\rangle$ for any partition $\lambda\vdash m$ and Schur functor $\bb S_{\lambda}$. We can also define arbitrary pullbacks up to replacing $T$ by an $\bb R$-linearly equivalent $\bb R$-Cartier $\bb R$-divisor. 
				
				Let $X$ be a projective variety. \cite{ELMNP} define augmented and diminished base loci for $\bb R$-Cartier $\bb R$-divisors by
				
				\[\bb B_+(D)=\bigcap_{D=A+E}{\rm Supp}\, E\quad\text{ and }\quad\bb B_-(D)=\bigcup_{D+A\text{ is a }\bb Q\text{-divisor}}\bb B(D+A).\]
				Note that the stable base locus $\bb B(D)$ is not as easy to define for $\bb R$-divisors. Natural candidates are $\bigcap_{D'\in|D|_{\bb R}}{\rm Supp}\, D'$ and $\bigcap_{D'\in|D|_{\bb Q}}{\rm Supp}\, D'=\bigcap_{m\geq 1}{\rm Bs}(\lfloor mD\rfloor)\bigcup\bigcap_{m\geq 1}{\rm Supp}(\{mD\})$. They  are not equal in general. 
				We list the basic results about augmented and diminished base loci from \cite{ELMNP}.
				
				\begin{remark}[\cite{ELMNP}]\label{rmk:RealB}
					Let $X$ be a projective variety.
					Below all divisors are $\bb R$-Cartier $\bb R$-divisors and $||\cdot||$ is any norm on the finite dimensional real vector space $N^1(X)$. Generally $D$ denotes an arbitrary divisor, while $A$ denotes an ample divisor.
					\begin{enumerate}
						\item $\bb B_{\pm}(D)$ are homogeneous numerical invariants of $D$.
						\item $\bb B_{\pm}(D)=\emptyset$ if and only if $D$ is ample (resp.~nef). At the other extreme $\bb B_{\pm}(D)=X$ if and only if $D$ is not big (resp.~not psef).
						\item There exists $\varepsilon>0$ such that $\bb B_+(D-D')\subseteq\bb B_+(D)$ for all $D'$ with $||D'||<\varepsilon$. Equality holds if additionally $D'$ is ample.
						\item $\bb B_+(D_1+D_2)\subseteq\bb B_+(D_1)\bigcup\bb B_+(D_2)$. 
						\item $\bb B_-(D)=\bigcup_A\bb B_+(D+A)$.
						If $\lim_{m\to\infty}A_m=0$, then $\bb B_-(D)=\bigcup_m\bb B_+(D+A_m)$.
						\item If $\bb B_-(D)$ is closed, then there exists $\epsilon>0$ such that $\bb B_-(D)=\bb B_+(D+A)$ for all $A$ with $||A||<\varepsilon$.
						\item There exists $\varepsilon>0$ such that $\bb B_-(D-A)=\bb B_+(D-A)=\bb B_+(D)$ if $||A||<\varepsilon$.
					\end{enumerate}
					
					To these we add
					\begin{enumerate}[(a)]
						\item $\bb B_{\pm}(D_1+D_2)\subseteq\bb B_{\pm}(D_1)\bigcup\bb B_-(D_2)$.
						\item In statement (3) we can replace $D'$ being ample with nef. 
						\item $\bb B_-(D)=\bigcup_A\bb B_-(D+A)=\bigcup_m\bb B_-(D+L_m)$ if $\lim_{m\to\infty}L_m=0$ is a sequence of nef divisors.
					\end{enumerate}
					\noindent \cite{ELMNP} assume that $X$ is also normal, but this does not appear to be used in the proofs of the results above.
					\begin{proof}
						(a) Use (3), (4), and (5). Then $\bb B_+(D_1+D_2)\subseteq\bb B_+(D_1-A)\bigcup\bb B_+(D_2+A)\subseteq\bb B_+(D_1)\bigcup\bb B_-(D_2)$ for sufficiently small ample $A$. The statement about $\bb B_-$ follows from this and (5).
						
						(b) Assume $D'$ is nef. 
						By (a), we have $\bb B_+(D)\subseteq\bb B_+(D-D')\bigcup\bb B_-(D')=\bb B_+(D-D')$. The reverse inclusion is the first statement of (3).
						
						(c) The inclusions $\bb B_-(D)\supseteq\bb B_-(D+A)$ and $\bb B_-(D)\supseteq\bb B_-(D+L_m)$ follow from (a) and (2). Clearly the same holds for the unions. Conversely, by (a), for every fixed ample $A$ we have $\bb B_+(D+A)\subseteq\bb B_-(D+\frac 12A)$ and $\bb B_+(D+A)\subseteq\bb B_-(D+L_m)$ whenever $A-L_m$ is ample. These give the necessary reverse inclusions via (5).
					\end{proof}
				\end{remark}
				
				Inspired by Proposition \ref{prop:imageBpm} we define $\bb B_{\pm}$ for twisted bundles.
				
				\begin{definition}\label{def:BQtwist}
					$\bb B_{\pm}(E\langle T\rangle)\coloneqq\pi\bigl(\bb B_{\pm}(\xi+\pi^*T)\bigr)$.
				\end{definition}

				\noindent The definition only depends on the numerical class of $T$.
				V-positivity and L-positivity for twisted bundles are defined as in Definition \ref{VandLPositiveDefinition}. 
				As a consequence of parts (a), (b), (c) of Remark \ref{rmk:RealB} we have the following:
				
				\begin{corollary}\label{cor:RealB+}
					Let $E$ be a vector bundle and let $T$ be an $\bb R$-Cartier $\bb R$-divisor. 
					\begin{enumerate}
						\item There exists $\varepsilon>0$ such that if $||T'||<\varepsilon$ on $X$, then $\bb B_+(E\langle T-T'\rangle)\subseteq\bb B_+(E\langle T\rangle)$. If $T'$ is ample (or just nef) on $X$, then equality holds.
						\item There exists a sequence of $\bb Q$-Cartier $\bb Q$-divisors $T_m$ such that $T-T_m$ are all ample, $\lim_{m\to\infty}T_m=T$, and $\bb B_+(E\langle T\rangle)=\bb B_+(E\langle T_m\rangle)$.
						\item There exists a sequence of $\bb Q$-Cartier $\bb Q$-divisors $T_m$ such that $T_m-T$ are ample, $\lim_{m\to\infty}T_m=T$, and $\bb B_-(E\langle T\rangle)=\bigcup_m\bb B_-(E\langle T_m\rangle)$.
					\end{enumerate} 
				\end{corollary}
				
				\begin{lemma}\label{5.2}
					Let $E$ be an $\bb R$-twisted bundle. For any sequence $\{A_m\}_m\subset N^1(X)$ of ample classes such that $\lim_{m\to\infty}A_m=0$, we have
					$$\mathbb{B}_-(E)=\bigcup_m \mathbb{B}_+(E\langle A_m\rangle).$$
					\begin{proof}
						There exists a sequence $\{D_m=a_m\xi+\pi^*A_m\mid 0<a_m\ll 1\}$ such that each $D_m$ is ample and $\lim_{m\to\infty}D_m=0$. By Remark \ref{rmk:RealB}.(5),  we have $\mathbb{B}_-(\xi)=\bigcup_m\mathbb{B}_+(\xi+D_m)=\bigcup_m\mathbb{B}_+(\xi+\frac{1}{1+a_m}\pi^* A_m)\supseteq\bigcup_m\bb B_+(\xi+\pi^*A_m)$. The inclusion is Remark \ref{rmk:RealB}.(a). For every $m$ we can find $N_m$ such that $\frac 1{1+a_m}A_m-A_{N_m}$ is ample. Then $\bb B_+(\xi+\frac1{1+a_m}\pi^*A_m)\subseteq\bb B_+(\xi+\pi^*A_{N_m})$, again by Remark \ref{rmk:RealB}.(a). This gives the reverse inclusion. Finally apply Definition \ref{def:BQtwist}.
					\end{proof}
				\end{lemma}
				
				\subsection{Base loci and pullback}

				We will make use of the following elementary observation: 
				
				\begin{lemma}\label{lem:settheory}
					Let 
					\[\xymatrix{A'\ar[r]_{f'}\ar[d]_{g'}&B'\ar[d]^{g}\\ A\ar[r]^f&B}\]
					be a commutative diagram of sets and let $S\subseteq B'$ be a subset. Then
					\begin{enumerate}
						\item $g'({f'}^{-1}\{S\})\subseteq f^{-1}\{g(S)\}$
						\item If the diagram is cartesian and $f$ is surjective, then equality holds.
					\end{enumerate}
				\end{lemma} 
				
				\begin{remark}\label{rmk:basepullback}Let $f:X\to Y$ be a morphism of projective varieties. For $E$ a vector bundle on $Y$ by pulling back the evaluation sequence we see that 
					\[\text{Bs}(f^*E)\subseteq f^{-1}\text{Bs}(E)\quad\text{ and }\quad\bb B(f^*E)\subseteq f^{-1}\bb B(E).\] 
					Equality holds if $f_*\cal O_X=\cal O_Y$, i.e., $f$ is an algebraic fiber space.
					For $\bb R$-classes $L\in N^1(Y)$ we also have
					\[\bb B_-(f^*L)\subseteq f^{-1}\bb B_-(L).\]
					(We have $\bb B_-(L)=\bigcup_{A\text{ ample, }L+A\ \bb Q\text{-divisor}}\bb B(L+A)$, hence $f^{-1}\bb B_-(L)=\bigcup_Af^{-1}\bb B(L+A)$. On the other hand $\bb B_-(f^*L)=\bigcup_{H\text{ ample, }f^*L+H\ \bb Q\text{-divisor}}\bb B(f^*L+H)$. For every $H$ ample on $X$ we can find $A$ ample on $Y$ such that $H-f^*A$ is an ample $\bb Q$-class. Then $\bb B(f^*L+H)\subseteq\bb B(f^*(L+A))\subseteq f^{-1}\bb B(L+A)$. )
					
					For $\bb R$-twisted bundles we obtain similarly
					\[\bb B_-(f^*E\langle f^*T\rangle)\subseteq f^{-1}\bb B_-(E\langle T\rangle)\]
					by using Definition \ref{def:BQtwist} together with Lemma \ref{lem:settheory} for the fiber product diagram
					\[\xymatrix{\bb P(f^*E\langle f^*T\rangle)\ar[r]\ar[d] & \bb P(E\langle T\rangle)\ar[d]\\ X\ar[r]^{f}&Y}\]
				\end{remark}

				\begin{proposition}\label{prop:restrictedbasepullback}Let $f:X\to Y$ be a surjective morphism of projective varieties. Assume that at least one of the following conditions holds:
					\begin{enumerate}
						\item $Y$ is smooth, or
						\item $Y$ is normal and $f$ is equidimensional.
					\end{enumerate}
					Let $E$ be a $\bb R$-twisted bundle on $Y$. Then
					\[\bb B_-(f^*E)=f^{-1}\bb B_-(E).\]
					In particular $E$ is V-psef if and only if $f^*E$ is V-psef.
				\end{proposition}
				\noindent The ($\bb Q$-twisted) line bundle case of (1) is proved with analytic methods in \cite{DL21}*{Lemma 2.2}. Our proof is algebraic.
				\begin{proof}
					The arbitrary rank case follows from the rank 1 case as in Remark \ref{rmk:basepullback}. We now focus on the rank 1 case. Assume that $L$ is an $\bb R$-Cartier $\bb R$-divisor. From Remark \ref{rmk:RealB}.(c) we may assume that $L$ is a $\bb Q$-Cartier $\bb Q$-divisor. By the homogeneity of $\bb B_-$ we may assume that $L$ is a Cartier divisor.
					We will make repeated use of an elementary consequence of Remark \ref{rmk:basepullback}: Let $g:Z\to X$ be a morphism from a projective variety. Then
					\begin{enumerate}[(i)]
						\item If $g$ is surjective, then the conclusion of $\bb B_-$ commuting with pullback holds for $(f\circ g,L)$ if and only if it holds for both $(f,L)$ and $(g,L)$.
						\item If $g$ is a closed immersion, then if the conclusion holds for $(f\circ g,L)$, then $\bb B_-(f^*L)\cap Z=f^{-1}\bb B_-(L)\cap Z$.
					\end{enumerate}
					We first reduce to the generically finite case for $f$, respectively the finite case when $f$ is equidimensional. Note that $X$ is covered by varieties $Z$ of dimension $\dim Y$ each of which surjects onto $Y$. 
					For instance take general complete intersections through any point $x\in X$. 
					Two subsets of $X$ are equal precisely when they have equal intersection with every such $Z$. Then apply observation (ii). 
					
					When $f$ is equidimensional we want to cover $X$ by $Z$ that are furthermore finite over $Y$. \cite{FLmorphisms}*{Lemma 4.9} proves that this can be achieved at least generally by complete intersections of high degree. The argument there can be tweaked to prove that furthermore we can construct such $Z$ through every $x\in X$ not just through general $x$. It boils down to proving that if $\cal Z_i\to\cal U_i$ are finitely many flat families of positive dimensional (irreducible) subvarieties of $X$ with $\cal Z_i$ mapping generically finitely onto its image in $X$, then for every ample divisor $H$ on $X$ there exists $d$ such that for every $x\in X$ there exists an element of $|dH|$ passing through $x$ and that meets all 
					subvarieties $(Z_i)_{u_i}$ properly.
					
					Next we prove the finite case for $f$ when $X$ and $Y$ are normal, and we prove the case where $f$ is birational with $Y$ smooth.
					\smallskip
					
					\noindent\textit{The finite case.} Assume that $f:X\to Y$ is a finite morphism between normal varieties.
					With $K(X)$ denoting the function field of $X$, let $K'$ be a Galois closure of $K(X)$ over $K(Y)$, and let $Z$ be the normalization of $X$ in $K'$ so that $K(Z)=K'$ and we have a finite map $g:Z\to X$ of normal varieties. Then $G\coloneqq \text{Gal}(K'/K(Y))$ acts on $Z$. The action restricts to a transitive action on the fibers of $f\circ g$ and $Y=Z/G$.
					By (i) we may assume that $f$ is the quotient map and $X=Z$.
					
					If $M$ is any line bundle on $Y$, then $f^*M$ is $G$-equivariant thus $G$ acts on the sections of $f^*M$. Let $A$ be an ample divisor on $Y$. Then $H\coloneqq f^*A$ is ample on $X$. If $x\not\in\bb B_-(f^*L)$, then for all $q$ we can find $m=m(q)$ such that $mqf^*L+mH=f^*(mqL+mA)$ has a section $s$ that does not vanish at $x$. Put $y\coloneqq f(x)$. Using the transitivity of the action of the finite group $G$ on $f^{-1}\{y\}$, we deduce that a general section of $f^*(mqL+mA)$ does not vanish at any point of $f^{-1}\{y\}$. The vanishing of such a section is a divisor $D\sim f^*(mqL+mA)$. Its cycle pushforward to $Y$ is an effective divisor linearly equivalent to $|G|(mqL+mA)$ that does not pass through $y$. In particular $y\not\in\bb B_-(L)$ giving the nontrivial inclusion $\bb B_-(f^*L)\supseteq f^{-1}\bb B_-(L)$.
					\smallskip
					
					Part (2) follows by using the reduction to the finite case and observation (i) for the normalizations $Z^{\nu}\to Z\to Y$. 
					\smallskip

					\noindent\textit{The birational case with $Y$ smooth.}
					Up to taking a resolution of $X$, using (i), we may assume that $X$ is also smooth. Let $x\not\in\bb B_-(f^*L)$. We aim to prove that $y\coloneqq f(x)$ is not in $\bb B_-(L)$. 
					
					Let $A$ be an ample divisor on $Y$. Up to scaling $A$, we may assume $f^*A=H+F$ for some ample divisor $H$ and effective divisor $F$ on $X$.
					Then $H_t\coloneqq f^*A-tF$ is ample for all $0<t<1$.
					Since the diminished base locus is independent of the polarization, for every $q$ and $t$ there exists $m=m(q,t)$ such that $x$ is not in the base locus of $mqf^*L+mH_t$. In particular there exists an effective divisor $D_m$ that does not pass through $x$ such that $D_m+mtF\sim f^*(mqL+mA)$. Thus the linear series $|f^*(mqL+mA)|$ contains effective divisors with multiplicity at most $mt\cdotp\mult_xF$.
					
					From the discussion above, by dividing by $m$ and then taking $t$ to 0 it follows that $|f^*(qL+A)|_{\bb Q}$ has (asymptotic) order $0$ at $x$. This order is a divisorial valuation on $K(X)=K(Y)$, the valuation corresponding to the exceptional divisor $E$ of blowing-up $x$. It has center $x$ on $X$ and $y= f(x)$ on $Y$. Since $f$ is an algebraic fiber space, we have $|f^*(qL+A)|_{\bb Q}=f^*(|qL+A|_{\bb Q})$, hence $|qL+A|_{\bb Q}$ also has order 0 at $y$. It follows from \cite{ELMNP}*{Proposition 2.8}, since $qL+A$ is big, that $y\not\in\bb B_-(qL+A)$, i.e., not in $\bb B_-(L+\frac 1qA)$ for any $q\geq 1$. But it is easy to see that $\bb B_-(L)=\bigcup_{q\geq 1}\bb B_-(L+\frac 1qA)$. In conclusion $y\not\in\bb B_-(L)$.
					\smallskip
					
					At this point, to prove (1), it is tempting to start from the Stein factorization $X\to Z\to Y$ and normalize $X^{\nu}\to Z^{\nu}\to Y$. However there does not appear a way to do so that produces $Z^{\nu}$ smooth. We used smoothness in the birational case for the application of \cite{ELMNP}*{Proposition 2.8}. Instead we will use flattenings.
					
					There exists a birational morphism $Y'\to Y$ such that if $X'$ denotes the unique irreducible component of $Y'\times_YX$ that dominates $Y'$, then the induced $f':X'\to Y'$ is flat, in particular finite. Let $\widetilde Y\to Y'$ be a resolution of singularities. Let $\widetilde X$ be the unique irreducible component of $\widetilde Y\times_{Y'}X'$ that dominates $\widetilde Y$. It is not clear that $\widetilde X\to\widetilde Y$ is flat, but it is finite. Finally let $\widetilde X^{\nu}$ denote the normalization. The induced $\widetilde X^{\nu}\to\widetilde Y$ is finite between normal varieties, while $\widetilde Y\to Y$ is birational with $Y$ smooth. From the finite and birational cases we deduce that the conclusion holds for $\widetilde X^{\nu}\to Y$ by observation (i). Then for the same reason it also holds for $X\to Y$.
				\end{proof}

				The corresponding statement for augmented base loci is slightly different.
				
				\begin{proposition}\label{prop:B+pull}
					Let $f:X\to Y$ be a surjective morphism of normal projective varieties. Denote by $\text{NF}(f)$ the \emph{nonfinite locus}, the union of curves $C\subseteq X$ contracted by $f$. Let $E$ be a ($\bb R$-twisted) vector bundle on $Y$. Then 
					\[\bb B_+(f^*E)=f^{-1}\bb B_+(E)\bigcup\text{NF}(f).\]
					In particular, if $f$ is generically finite, then $E$ is V-big if and only if $f^*E$ is V-big.
				\end{proposition}
				
				\noindent The finite case for ($\bb Q$-twisted) line bundles also appears in \cite{DL21}*{Lemma 2.3}. The birational case is treated in \cite{BBP}*{Proposition 2.3}. See also the subsequent discussions in \cites{BCL,Lopez}. 
				
				\begin{proof}
					Note that $NF(f)$ respects base change. 
					As in Proposition \ref{prop:restrictedbasepullback}, we reduce to the $\bb R$-Cartier $\bb R$-divisor case $L$.
					From Remark \ref{rmk:RealB}.(a) we may assume that $L$ is a $\bb Q$-Cartier $\bb Q$-divisor and by the homogeneity of $\bb B_+$ we may assume that $L$ is Cartier.
						If $f$ is not generically finite, then $f^*L$ is not big and $NF(f)=X$. The conclusion again follows easily.
					Assume henceforth that $f$ is generically finite.
					
					Let $X\overset{g}{\to}Z\overset{h}{\to}Y$ be the Stein factorization of $f$ with $Z$ normal. Then $g$ is birational and $h$ is finite, Furthermore $NF(f)=NF(g)$.
					Fix $A$ an ample divisor on $Y$. Then $h^*A$ is also ample. for all large enough $m$ we have
					\[\bb B_+(h^*L)=\bb B_-(h^*L-\frac 1mh^*A)=h^{-1}\bb B_-(L-\frac 1mA)=h^{-1}\bb B_+(L).\]
					We have used Remark \ref{rmk:RealB}.(7) for the first and last equality, and the finite case of Proposition \ref{prop:restrictedbasepullback} for the middle equality. Thus we have reduced to the birational case for $f$. At this point one can cite \cite{BBP}*{Proposition 2.2}, but we give a proof here.
					
					If $H$ is ample on $X$, and $\epsilon>0$ then $f^*L-\epsilon H$ is negative on all curves $C$ contracted by $f$. We deduce that $\text{NF}(f)\subseteq\bb B_+(f^*L)$.
					If $L$ is nef, then $\bb B_+(L)$ is the union of subvarieties $V$ of $Y$ such that $(L^{\dim V}\cdot V)=0$ by the main result of \cite{Nakamaye} in the smooth case, and by \cite{Birkar}*{Theorem 1.4} in the general projective case over any field.
					After applying this to $f^*L$, the result follows from the projection formula. In particular, if $A$ is ample on $Y$, then $\bb B_+(f^*A)=\text{NF}(f)$.
					
					Let $x\not\in\text{NF}(f)$. From the previous paragraph we can write $f^*A=H+F$ where $H$ is an ample $\bb Q$-divisor on $X$ and $F$ is effective, with $x\not\in\text{Supp}(F)$. For all large $m$ we have
					\[\bb B_+(f^*L)=\bb B_-(f^*L-\frac 1mH)\subseteq\bb B_-(f^*(L-\frac 1mA))\bigcup F\subseteq f^{-1}\bb B_-(L-\frac 1mA)\bigcup F=f^{-1}\bb B_+(L)\bigcup F.\]
					We used Remark \ref{rmk:RealB}.(7) for the first and last equalities. The first inclusion comes from Lemma \ref{lem:tensorbasepm}, and the second from Remark \ref{rmk:basepullback}.
					As we vary $x$, we obtain $\bb B_+(f^*L)\subseteq f^{-1}\bb B_+(L)\bigcup\text{NF}(f).$
					It is sufficient to prove $f^{-1}\bb B_+(L)\subseteq\bb B_+(f^*L)$.
					
					Since $f$ is birational and $X$ and $Y$ are normal, $f$ is a fiber space. 
					Let $H$ and $A$ be ample Cartier divisors on $X$ and $Y$ respectively such that $H-f^*A$ is ample.
					Then for sufficiently large $m$, we have 
					\[f^{-1}\bb B_+(L)=f^{-1}\bb B(L-\frac 1mA)=\bb B(f^*L-\frac 1mf^*A)\subseteq\bb B(f^*L-\frac 1mH)\bigcup\bb B(\frac 1m(H-f^*A))=\bb B(f^*L-\frac 1mH)=\bb B_+(f^*L).\]
				\end{proof}

				\begin{corollary}[Homogeneity for $\bb R$-twists]
					Let $T$ be a $\bb R$-Cartier $\bb R$-divisor and $E$ a vector bundle on the normal projective variety $X$. Then $\bb B_{\pm}(E\langle T\rangle)=\bb B_{\pm}(S^c(E\langle T\rangle))=\bb B_{\pm}(\bigotimes^c(E\langle T\rangle))$ for all $c\geq 1$.
				\end{corollary}
				\begin{proof}
					This is not a formal consequence of Definition \ref{def:BQtwist} and of the rank 1 case because we work on different projectivizations. 
					The statement about $\bb B_-$ follows from the corresponding statement about $\bb B_+$ by Remark \ref{rmk:RealB}.(5).
					
					Fix $c$. We can find simultaneous rational approximations of $T$ as in Corollary \ref{cor:RealB+} for $E\langle T\rangle$, for $S^c(E\langle T\rangle)=(S^cE)\langle cT\rangle$, and for tensor powers. Thus we can assume that $T$ is a $\bb Q$-Cartier $\bb Q$-divisor. 
					
					Let $n$ be such that $nT$ is Cartier. Let $f:Y\to X$ be a Gieseker cover (cf.~\cite{L}*{Theorem 4.1.10}), a finite flat cover such that $f^*\cal O_X(nB)\simeq \cal O_Y(nB')$ for some actually $\bb Z$-Cartier divisor $B'$ on $Y$. We replace $Y$ by its normalization. It is not clear that flatness is preserved, but the other properties are.
					Proposition \ref{prop:B+pull} reduces the problem to the untwisted bundle case which is Corollary \ref{cor:homogeneousB}.
				\end{proof}

				Similarly Lemma \ref{lem:tensorbasepm} and Lemma \ref{lem:directsums} extend to twisted bundles.
				
				\begin{proposition}
					Let $X$ be a normal projective variety.
					\begin{enumerate}[(i)]
						\item If $E_1\langle T_1\rangle$ and $E_2\langle T_2\rangle$ are two $\mathbb{R}$-twisted V-pseudo-effective vector bundles on $X$, then $E_1\langle T_1\rangle\otimes E_2\langle T_2\rangle$ is also V-psef. If furthermore one of them is in fact V-big, then so is the tensor product.
						
						\item  Let $E_1\langle T\rangle$ and $E_1\langle T\rangle$ be two $\mathbb{R}$-twisted vector bundles on $X$. Then $E_1\langle T\rangle\oplus E_2\langle T\rangle$ V-positive if and only if $E_1\langle T\rangle$ and $E_2\langle T\rangle$ are V-positive.
					\end{enumerate} 
					\begin{proof}
						$(i)$ Let $A_{m,1}$ and $A_{m,2}$ be ample $\bb R$-Cartier $\bb R$-divisors that limit to 0 zero with $m$ such that $T_i+A_{m,i}$ are $\bb Q$-Cartier $\bb Q$-divisors. It is sufficient to prove that $E_1\langle T_1+A_{m,1}\rangle\otimes E_2\langle T_2+A_{m,2}\rangle=(E_1\otimes E_2)\langle T_1+A_{m,1}+T_2+A_{m,2}\rangle$ is V-big. For fixed $m$, we pass to a finite Gieseker cover $f:Y\to X$ such that such that $f^*(T_i+A_{m,i})$ is $\bb Q$-linearly equivalent to a Cartier divisor for $i\in\{1,2\}$. Since $\bb B_{\pm}$ are numerical invariants, and by Proposition \ref{prop:B+pull}, we reduce to the case where the twists are Cartier, i.e., to the untwisted bundle case which is Lemma \ref{lem:tensorbasepm}. Part $(ii)$ is similar. One reduces to the untwisted case, which is Lemma \ref{lem:directsums}.
					\end{proof}
				\end{proposition}

			\section{L-big and L-psef bundles}
			
			\begin{lemma}Let $E$ be a vector bundle on $X$. By a slight abuse denote $L=\cal O_{\bb P(E)}(1)$. Then
				\begin{enumerate}[(i)]
					\item The following are equivalent
					\begin{enumerate}
						\item $E$ is L-big.
						\item For every ample divisor $A$ on $\bb P(E)$ there exists $N>0$ such that for all $n\geq N$ and all sufficiently large $m$ (depending on $n$), the line bundle $L^{\otimes nm}(-mA)$ is effective.
						\item There exists $A$ ample on $\bb P(E)$ and $n>0$ such that $L^{\otimes n}(-A)$ is pseudeo-effective .
						\item For every ample divisor $H$ on $X$ there exists $N>0$ such that for all $n\geq N$ and all sufficiently large $m$ (depending on $n$), the bundle $S^{nm}E(-mH)$ is effective.
						\item There exists $H$ ample on $X$ and $n>0$ such that $S^{n}E(-H)$ is pseudo-effective.
					\end{enumerate}
					\smallskip
					
					\item The following are equivalent
					\begin{enumerate}
						\item $E$ is L-psef.
						\item For some (or every) ample divisor $A$ on $\bb P(E)$ and every $n>0$, the line bundle $L^{\otimes nm}(mA)$ is effective for all sufficiently large $m$ (depending on $n$).
						\item (In the smooth case) There exists an ample divisor $A$ on $\bb P(E)$ such that $L^{\otimes m}(A)$ is effective for all sufficiently large $m$.
						\item For some (or every) ample divisor $H$ on $X$ and every $n>0$ the bundle $S^{nm}E(mH)$ is effective for all sufficiently large $m$ (depending on $n$).
						\item (In the smooth case) There exists an ample divisor $H$ on $X$ such that $S^mE(H)$ is effective for all sufficiently large $m$.
					\end{enumerate}
				\end{enumerate}
			\end{lemma}
			
			\noindent The result is valid without restriction on the characteristic of the base field.
			
			\begin{proof}
				$(i)$. It is clear that $(a)-(c)$ are restatements of $L$ being a big line bundle. 
				We now relate the conditions in $(d)-(e)$ to positivity conditions on $L$.
				The effectivity of $S^mE(-H)$ is equivalent to that of $L^{\otimes m}\otimes\pi^*\cal O_X(-H)$.
				Let $\xi=c_1(L)$. Since $\xi$ is $\pi$-ample, we have that $\xi+n\pi^*H$ is ample for all large enough $n$. It follows that the cone $\langle\xi,\pi^*H\rangle$ slices through ${\rm Big}(\bb P(E))$ (it cannot be fully contained in its boundary because it contains an ample class). Any of the conditions $(d)-(e)$ imply that $\xi$ is effective (up to scaling) and cannot be on the boundary of the big cone, hence it is big. 
				
				Conversely let's prove $(b)\to(d)$. Fix $H$ ample on $X$ and choose $N$ such that $A=L\otimes\pi^*\cal O_X(NH)$ is ample and $(n-1)H$ is effective for all $n\geq N$.
				For $n\geq N$ the effectivity of $L^{\otimes nm}(-mA)$ is equivalent to that of $S^{(n-1)m}E(-NmH)$. Then for $n\geq N-1$ we obtain that $S^{nm}E(-mH)$ is effective.
				$(d)\to(e)$ is clear.
				\smallskip
				
				$(ii)$. Except for $(c)$ and $(e)$ which use the extra smoothness assumption, the equivalences are similar to those in part $(i)$. It is clear that both $(c)$ and $(e)$ imply that $L$ is pseudo-effective, hence $(a)$. For $(a)\to(c)$, we use \cite{Nak}*{Theorem V.1.12} in characteristic zero and \cite{CHMS} in positive characteristic. These handle the case where $L$ has positive numerical dimension. When its numerical dimension is zero, then $L$ is in fact effective and the implications are trivial.
				
				For $(c)\to (e)$, fix $A$ as in $(d)$. Then $A=L^{\otimes a}\otimes\pi^*\cal O_X(B)$ for some $a>0$ and some divisor $B$ on $X$. Choose $H$ ample on $X$ so that $H-B$ is effective. Then the effectivity of $L^{\otimes m}(A)$ implies that of $S^{m+a}E(H)$. 
			\end{proof}
			
			\begin{proposition}\label{prop:subLbig}
				Let $E'\subset E$ be an inclusion of vector bundles. If $E'$ is L-positive, then so is $E$.
			\end{proposition}
			
			\noindent We do not assume that $E'$ is a subbundle, meaning that $E/E'$ is a vector bundle.
			
			\begin{proof}For every ample $H$ we have inclusions $S^{mn}E'(\pm nH)\subset S^{mn}E(\pm nH)$.
			\end{proof}
			
			\begin{proposition}\label{prop:directLhomogeneity}
				If $E$ is L-positive then $S^cE$ is L-positive for al $c\geq 1$.
			\end{proposition}
			\begin{proof}The surjection $T^{mc}E\twoheadrightarrow S^{mc}E$ induces a sujection $S^mS^cE\twoheadrightarrow S^{mc}E$. By dualizing the corresponding surjection for the duals, in characteristic zero we obtain an inclusion $S^{mc}E\subset S^mS^cE$. This easily implies that if $E$ is L-big (resp.~L-psef), then so is $S^cE$.
			\end{proof}
			
			The converse is not true. This contradicts the intuition coming from line bundles, from V-positivity notions, or from other positivity notions such as ampleness or nefness that are also defined by the positivity of $\cal O_{\bb P(E)}(1)$.
			
			\begin{example}\label{ex:Lcounter}
				Choose three general cubics in $\bb P^2$. They induce an exact sequence
				$0\to M\to \cal O_{\bb P^2}^{\oplus 3}\to \cal O_{\bb P^2}(3)\to 0$.
				Let 
				\[E\coloneqq M^{\vee}(-1).\]
				We prove that $E$ is not L-psef, but $S^2S^{2m}E$ is L-big for all $m\geq 1$. This logically implies that the converse of Proposition \ref{prop:directLhomogeneity} (for L-bigness or L-pseudo-effectivity) fails either for the bundle $S^{2m}E$ and $c=2$, or fails for $E$ and $c=2m$.
				
				For the claim on the L-bigness, note that ${\rm rk}\, E=2$ and so for all $m\geq 1$ we have natural inclusions $(\bigwedge^2E)^{\otimes 2m}\subset S^2S^{2m}E$. See \cite{FLstability}*{Corollary 4.10}. One computes that $\bigwedge^2E=\det E=\cal O_{\bb P^2}(1)$ is big, hence so are all $S^2S^{2m}E$ by Proposition \ref{prop:subLbig}.
				
				For the failure of L-pseudo-effectivity, it is sufficient to check that for all $n>1$ and all $\ell>0$, the bundle $(S^{n\ell}E)(\ell)$ has no nonzero sections. We have exact sequences
				\[0\to S^{n\ell-1}(\cal O_{\bb P^2}^{\oplus 3}(-1))(\ell-4)\to S^{n\ell}(\cal O_{\bb P^2}^{\oplus 3}(-1))(\ell)\to S^{n\ell}E(\ell)\to 0.\]
				The totally split bundles on the left and in the middle of the exact sequence are direct sums of line bundles of degrees $2\ell-n\ell-4$ and $\ell-n\ell$ respectively. Since the intermediary cohomology of all line bundles on $\bb P^2$ in vanishes, we deduce that $H^0(\bb P^2,S^{n\ell}E(\ell))=0$ for all $n>1$ and $\ell>0$.\qed   
			\end{example}
			
			\begin{proposition}
				Let $E$ be a L-psef vector bundle and $F$ be an L-big bundle on $X$. Then $ E\otimes F$ is an L-big vector bundle. 
			\end{proposition}
			
			\begin{proof}The natural surjection $T^m(E\otimes F)=T^mE\otimes T^mF\twoheadrightarrow S^mE\otimes S^mF$ descends to a surjective $S^m(E\otimes F)\twoheadrightarrow S^mE\otimes S^mF$. By dualizing the corresponding surjection for the duals we also obtain a natural inclusion $S^mE\otimes S^mF\subset S^m(E\otimes F)$ in characteristic zero. Thus we also have inclusions $(S^mE(H))\otimes(S^mF(-2H))\subset S^m(E\otimes F)(-H)$ for any ample $H$ on $X$.
			\end{proof}
			
\section{Cones of big and of pseudoeffective vector bundles}

Let $X$ be a complex projective manifold of dimension $n$.
Let $N^*(X)=\bigoplus_{i=0}^nN^i(X)$ be the real numerical ring of $X$.
Let $E$ be a vector bundle. Let $\ch(E)$ be its {\bf Chern character}, and also consider its {\bf log-Chern character},
\[\lc(E)\coloneqq\log\ch(E)=\log r+\frac 1rc_1(E)-\frac1{2r^2}(2rc_2(E)-(r-1)c_1^2)+\ldots\in N^*(X).\]
We have $\ch(E\otimes F)=\ch(E)\cdot\ch(F)$ and $\ch(E\oplus F)=\ch(E)+\ch(F)$. Furthermore $\lc(E\otimes F)=\lc E+\lc F$.
\cite{Fu} considered the convex cones ${\rm Amp}_{\lc}(X)\subset\Nef_{\lc}(X)\subseteq N^*(X)$ generated by $\lc(E)$ where $E$ is ample, respectively nef. The first author then proved that the two cones are the same up to closure. A related construction was carried out using the Chern character $\ch$.

One can similarly define positive cones of V-positive bundles using characteristic classes of V-big and respectively V-pseudoeffective bundles. 
In this way we obtain ${\rm Big}^V_{\lc}(X)$ and ${\rm Psef}^V_{\lc}(X)$.
As in \cite{Fu} one can prove that they are full-dimensional cones in $N^*(X)$ that agree up to closure and surject onto the classical cones ${\rm Big}(X)\subset\Nef(X)\subset N^1(X)$ via projection onto the degree 1 component $N^*(X)\to N^1(X)$. This is an application of Proposition \ref{prop:goodforcones}. 

These cones exhibit pathological behavior. For instance they are not necessarily pointed cones, ample divisors do not determine interior classes, and the big cone is not necessarily open.

Many of the classical positivity notions for divisors are numerical, e.g., ampleness, nefness, bigness, pseudoeffectivity. This is not the case in higher rank. For instance $\cal O_{\bb P^1}(1)\oplus\cal O_{\bb P^1}(-1)$ and $\cal O_{\bb P^1}^{\oplus 2}$ have the same classes, but the first is not V-positive.

	\end{document}